\documentclass{amsart}
\usepackage{hyperref}
\usepackage{graphicx}

\newcommand{\Z}{\mathbf{Z}}
\newcommand{\N}{\mathbf{N}}

\newcommand{\cF}{\mathcal{F}}

\newcommand{\cI}{\mathcal{I}}

\newcommand{\eps}{\epsilon}

\DeclareMathOperator{\E}{\mathbf{E}}
\let\P\relax 
\DeclareMathOperator{\P}{\mathbf{P}}

\theoremstyle{plain}
\newtheorem{theorem}{Theorem}

\newtheorem{lemma}[theorem]{Lemma}

\theoremstyle{definition}

\theoremstyle{remark}
\newtheorem{remark}[theorem]{Remark}

\begin{document}
\title{A Mass Transport Proof of the Ergodic Theorem}
\author{Calvin Wooyoung Chin}
\date{\today}
\begin{abstract}
It is known that a gambler repeating a game with positive expected value has a positive probability to never go broke.
We use the mass transport method to prove the generalization of this fact where the gains from the bets form a stationary, rather than an i.i.d., sequence.
Birkhoff's ergodic theorem follows from this by a standard argument.
\end{abstract}
\maketitle

Let $X_1,X_2,\ldots$ be (real-valued) random variables, and write
\[ S_n := X_1+\cdots+X_n \qquad \text{for all $n\in\N$.} \]
In this note, we write $\N := \{1,2,3,\ldots\}$.

It is known \cite{gambler} that a gambler that repeats a game with positive expected value has a positive chance to never go broke.
More precisely, if $X_1,X_2,\ldots$ are i.i.d.\ and $\E X_1 > 0$, then
\[
\P(S_n > 0 \text{ for all } n\in\N) > 0.
\]
In this note, we show that this holds whenever $X_1,X_2,\ldots$ are stationary and $\E X_1 > 0$ (Lemma~\ref{lem:max_ergo_gamb}) using the ``mass transport" method.
Although no knowledge of the method will be needed, one might want to take a look at the short paper \cite{Hag99} to get a flavor of mass transport.

Birkhoff's ergodic theorem (Theorem~\ref{thm:birk_ergo}) then follows by a standard argument.
We include the derivation of the ergodic theorem for completeness.

\begin{lemma} \label{lem:max_ergo_gamb}
If the sequence $X_1,X_2,\ldots$ is stationary and $\E X_1 > 0$, then
\[ \P(S_n > 0 \text{ for all } n\in\N) > 0. \]
\end{lemma}

\begin{proof}
By Kolmogorov's extension theorem, we may assume without loss of generality the existence of a doubly-infinite stationary sequence
\[ \dots,X_{-2},X_{-1},X_0,X_1,X_2,\dots. \]
Define $S_n$ for every $n\in\Z$ so that $S_0 = 0$ and $S_n - S_{n-1} = X_n$ for all $n\in\Z$.

Let $n \in \Z$.
We call $m \in \Z$ a \emph{record after} $n$ if
\[ m > n \qquad \text{and} \qquad S_m = \min\{S_{n+1},S_{n+2},\ldots,S_m\}. \]
Notice that $n+1$ is always a record after $n$.
We now introduce the ``mass" $M(n,m)$ that $n$ sends to each $m\in\Z$ as follows.
If $X_{n+1} \le 0$, then $M(n,m) := 0$ for all $m\in\Z$.
If $X_{n+1} > 0$, then let $n+1 = n_0 < n_1 < n_2 < \cdots$ (this might be finite in length) be the enumeration of the records after $n$, and let
\begin{equation} \label{eq:mass_trans_defn}
M(n,n_j) := \max\{S_{n_{j-1}}, S_n\} - \max\{S_{n_j}, S_n\} \qquad \text{ for each $j\in\N$.}
\end{equation}
For all $m \in \Z \setminus \{n_1,n_2,\ldots\}$, let $M(n,m) := 0$.
See Figure~\ref{fig:transport} for a visualization of the mass transport.

\begin{figure}
\includegraphics{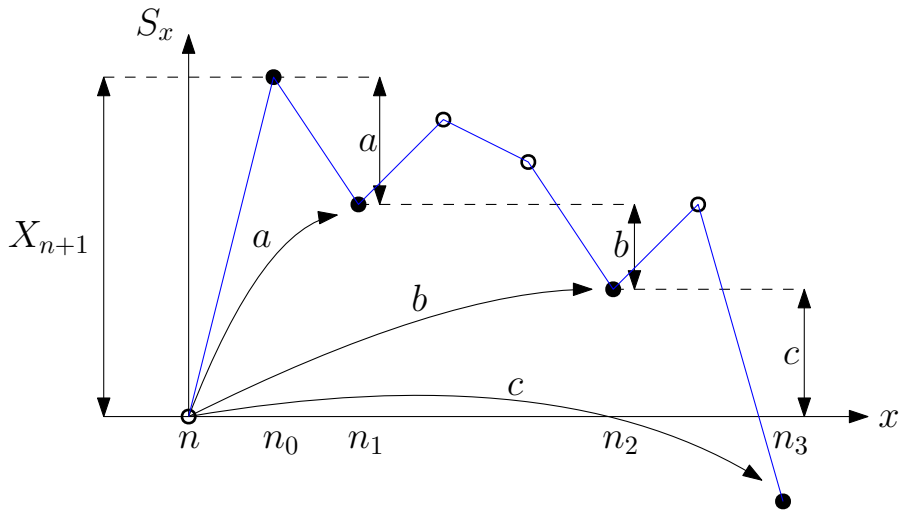}
\caption{
The records after $n$ (filled dots) and the mass that $n$ sends to them.
In this particular case, $n$ sends out $a+b+c$ amount of mass, which equals $X_{n+1}$.}
\label{fig:transport}
\end{figure}

Since $\dots,X_{-1},X_0,X_1,\dots$ is stationary and $M(n,m) \ge 0$ for all $n,m \in \Z$, we have
\begin{equation} \label{eq:erg_mt}
\E\biggl[ \sum_{n=1}^\infty M(0,n) \biggr]
= \sum_{n=1}^\infty \E[M(0,n)] = \sum_{n=1}^\infty \E[M(-n,0)]
= \E\biggl[ \sum_{n=1}^\infty M(-n,0) \biggr].
\end{equation}
This simple equality is at the heart of the mass transport method.
Suppose that $\P(S_n \le 0 \text{ for some }n\in\N) = 1$.
We will derive a contradiction by evaluating each side of \eqref{eq:erg_mt}.

Assume that $X_1 > 0$, and let $\tau := \inf\{n\in\N : S_n \le 0\}$.
Notice that $\tau < \infty$ a.s.
Let $1 = n_0 < n_1 < \cdots < n_t = \tau$ be the records after $0$ up to $\tau$.
Then we have
\[
\sum_{n=1}^\infty M(0,n) = \sum_{j=1}^t M(0,n_j)
= \sum_{j=1}^{t-1} (S_{n_{j-1}}-S_{n_j}) + (S_{n_{t-1}} - S_0)
= X_1.
\]
Thus, the left side of \eqref{eq:erg_mt} equals $\E[X_1;X_1>0]$.

Let us examine the sum on the right side of \eqref{eq:erg_mt}.
See Figure~\ref{fig:received} for a visualization.
If $X_0 > 0$, then $0$ is not a record after any number below $-1$, and thus the sum is $0$.
Assume $X_0 \le 0$, and let $-1 = m_0 > m_1 > m_2 > \cdots$ (which might be finite in length) be the enumeration of the numbers $m < 0$ such that
\[
S_m < \min\{S_{m+1},\ldots,S_{-1}\}.
\]

Let $m < 0$ and let us compute $M(m,0)$.
First assume that $m=m_j$ for some $j\in\N$, and consider the cases (a) $S_{m_{j-1}} < 0$ and (b) $S_{m_{j-1}} \ge 0$.
If (a) is the case, then $0$ is not a record after $m_j$, and thus $M(m,0)=0$.
Assume that (b) is the case.
By the definition of $m_0,m_1,\dots$, we have $S_n \ge S_{m_{j-1}}$ for all $n=m_j+1,\ldots,m_{j-1}$. This implies that $m_{j-1}$ is a record after $m_j$.
Since
\[ 0 \le S_{m_{j-1}} < \min\{S_{m_{j-1}+1},\ldots,S_{-1}\}, \]
we see that $m_{j-1}$ and $0$ are consecutive records after $m_j$.
Thus,
\[ M(m,0) = S_{m_{j-1}} - \max\{0,S_{m_j}\}. \]
Combining (a) and (b) yields
\[
M(m,0) = \max\{S_{m_{j-1}},0\} - \max\{S_{m_j},0\}.
\]

If $m \ne m_j$ for all $j\ge0$, and $k\ge0$ is the largest number such that $m<m_k$, then $S_m \ge S_{m_k}$.
Since $m_k$ is a record after $m$, the definition of $M$ tells us that $M(m,0) = 0$; the two maximums in \eqref{eq:mass_trans_defn} are both $S_m$ even if $0$ is a record after $m$.
We now know what $M(m,0)$ is for all $m < 0$, and this yields
\[
\begin{split}
\E\biggl[ \sum_{n=1}^\infty M(-n,0) \biggr]
&= \E\biggl[ \sum_{j\ge 1} \bigl(\max\{S_{m_{j-1}},0\} - \max\{S_{m_j},0\}\bigr); X_0\le 0 \biggr] \\
&\le \E[-X_0;X_0\le 0].
\end{split}
\]

\begin{figure}
\includegraphics{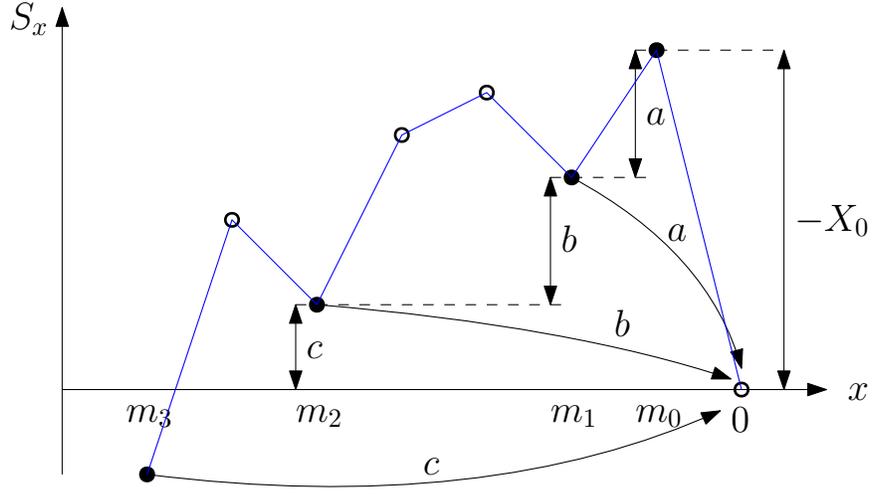}
\caption{
The mass $0$ receives when $X_0 \le 0$.
In this particular case, $0$ receives $a+b+c = -X_0$, but it might receive less if $\inf_{n<0} S_n > 0$.
}
\label{fig:received}
\end{figure}

The equation \eqref{eq:erg_mt} now gives
\[ \E[X_1;X_1>0] \le \E[-X_0;X_0\le 0], \]
which implies $\E X_1 \le 0$;
this is a contradiction.
\end{proof}

\begin{remark}
Our argument actually proves the maximal ergodic theorem \cite[Theorem~24.2]{BilPM}, which says that if $X_1,X_2,\ldots$ is a stationary sequence, then
\[
\E[X_1; S_n \le 0 \text{ for some }n\in\N] \le 0.
\]
Indeed, the left side of \eqref{eq:erg_mt} is bounded below by
\[
\E[X_1; X_1 > 0 \text{ and } S_n\le 0 \text{ for some } n\in\N],
\]
while the right side of \eqref{eq:erg_mt} is bounded above by $\E[-X_0;X_0\le 0]$.
Thus,
\[ \E[X_1; X_1 > 0 \text{ and } S_n\le 0 \text{ for some } n\in\N]
\le \E[-X_1;X_1\le 0], \]
and therefore $\E[X_1; S_n\le 0 \text{ for some } n\in\N] \le 0$.

There is a short proof of the maximal ergodic theorem which however lacks in intuition; see \cite[Theorem~24.2]{BilPM}, for example.
Our proof is an attempt to remedy this problem by utilizing the intuitive principle of mass transport.
\end{remark}

We now prove Birkhoff's ergodic theorem.
Let $(\Omega,\cF,\P)$ be the underlying probability space, and a measurable map $T \colon \Omega \to \Omega$ be \emph{measure-preserving} in the sense that $\P(T^{-1}A) = \P(A)$ for all $A \in \cF$.
An event $A$ is \emph{invariant} under $T$ if $T^{-1}A = A$, and we denote the $\sigma$-field of all invariant events by $\cI$.

\begin{theorem}[Birkhoff's ergodic theorem] \label{thm:birk_ergo}
Let $X_1$ be a random variable with finite mean, and write $X_n := X_1\circ T^{n-1}$ for $n=2,3,\dots$.
If $S_n := X_1+\cdots+X_n$ for all $n\in\N$, then
\[ S_n/n \to \E[X_1 \mid \cI] \qquad \text{a.s.} \]
\end{theorem}

\begin{proof}
Since $\E[X_1\mid\cI] \circ T = \E[X_1\mid\cI]$, we may assume that $\E[X_1\mid\cI] = 0$.
Let $\eps > 0$ and $A_\eps := \{\liminf_{n\to\infty} S_n/n < -\eps\}$.
On the event $A_\eps$, we have $S_n+n\eps < 0$ for some $n\in\N$.
Thus,
\begin{equation} \label{eq:ergo_prob_zero}
\P((S_n+n\eps)1_{A_\eps} > 0 \text{ for all }n\in\N) = 0.
\end{equation}
Since $A_\eps \in \cI$, we have
\[ ((X_1+\eps)1_{A_\eps}) \circ T^{n-1} = (X_n+\eps)1_{A_\eps} \qquad \text{for all $n\in\N$.} \]
As $((X_n+\eps)1_{A_\eps})_{n\in\N}$ is a stationary sequence, \eqref{eq:ergo_prob_zero} and Lemma~\ref{lem:max_ergo_gamb} imply
\[ \E[(X_1+\eps)1_{A_\eps}] \le 0. \]
Since $\E[X_1\mid\cI] = 0$, we have
\[ \E[(X_1+\eps)1_{A_\eps}] = \eps\P(A_\eps), \]
and thus $\P(A_\eps) = 0$.
As $\eps$ is arbitrary, we have $\liminf_{n\to\infty} S_n/n \ge 0$.
Applying this result to $-X_1$ in place of $X_1$ gives $\limsup_{n\to\infty} S_n/n \le 0$.
Therefore, we have $S_n/n \to 0$ a.s.
\end{proof}

\bibliographystyle{alpha}
\bibliography{references}
\end{document}